\documentclass[10pt,letterpaper]{article}
\pdfpagewidth=\paperwidth
\pdfpageheight=\paperheight
\usepackage{amsmath,amsthm}
\usepackage{url}

\makeatletter

\renewcommand\section{\@startsection {section}{1}{\z@}
{-30pt \@plus -1ex \@minus -.2ex}
{2.3ex \@plus.2ex}
{\normalfont\normalsize\bfseries\boldmath}}

\renewcommand\subsection{\@startsection{subsection}{2}{\z@}
{-3.25ex\@plus -1ex \@minus -.2ex}
{1.5ex \@plus .2ex}
{\normalfont\normalsize\bfseries\boldmath}}

\renewcommand{\@seccntformat}[1]{\csname the#1\endcsname. }

\makeatother

\newtheorem{theorem}{Theorem}
\newtheorem{corollary}{Corollary}

\theoremstyle{definition}
\newtheorem{conjecture}{Conjecture}

\begin{document}
\newcommand{\rad}{\textrm{rad}}
\def\uplim{{700}}

\begin{center}
\uppercase{\bf No new Goormaghtigh primes up to $10^{\uplim}$}
\vskip 20pt
{\bf Jon Grantham}\\ 
{\small Institute for Defense Analyses/Center for Computing Sciences, Bowie, Maryland}\\
{\tt grantham@super.org}\\ 
\end{center}
\vskip 20pt

\centerline{\bf Abstract}
\noindent
The Goormaghtigh conjecture states that the only two numbers which have
two non-trivial representations as repunits are $31$ and $8191$. We call 
such a prime number a {\it Goormaghtigh prime}.
We show that there are no other Goormaghtigh primes less than $10^{\uplim}$.

\pagestyle{myheadings}
\thispagestyle{empty}
\baselineskip=12.875pt
\vskip 30pt

\noindent\textbf{Correction notice.} The published version contained two independent errors.  The
program used for the bounded search transposed the base and exponent indices
in its modular lookup and did not implement the required search domain.
Separately, the proof of Theorem~2 did not establish its abc-conditional
conclusion.  The present version replaces both the computation and the
proof.

\section{Introduction}

Recall that a repunit is a positive integer whose digital representation in some base $b$ consists only of $1$s. Goormaghtigh \cite{Goormaghtigh} observed that the only numbers up to $10^5$ with
two non-trivial representations as repunits are $31$ (bases $2$ and $5$) and $8191$ (bases $2$ and $90$). The conjecture
that these are the only two such numbers has become known as the ``Goormaghtigh
Conjecture.'' More precisely, 
\begin{conjecture}
The only solutions to 
\begin{equation}
\label{eq:goorm}
N=\frac{x^m-1}{x-1}=\frac{y^n-1}{y-1}
\end{equation} with
integers $y>x\ge 2$, $m>n>2$ are $N=31$ and $N=8191$.
\end{conjecture}

We exclude the case $n=2$ in Conjecture 1 because every integer $N$ is a length-$2$ repunit in base $N-1$.
It is currently unknown whether the equation has finitely many solutions. Some results are known for small exponents and the case $\gcd(m-1,n-1)>1$, which we will use below.
Note that $31$ and $8191$ are both primes. We propose the terms {\it Goormaghtigh
primes} for such repunits $N$ that are primes and {\it Goormaghtigh numbers}
for any $N$ with two representations.
Prime repunits have been also studied as Brazilian primes; see \cite{brazil}. A Goormaghtigh prime can also be characterized as a Brazilian prime to two different
bases.
Bateman and Stemmler \cite{MR0138616}, using computations of Horn, noted that the only $N$ less than $1.275\times 10^{10}$ satisfying \eqref{eq:goorm} with $x$, $y$ and $N$ all prime is $31$.
In this paper, we look at a condition weaker than in the Goormaghtigh conjecture,
but stronger than in the Bateman-Stemmler question, by searching for Goormaghtigh
primes.
We use recent results of Bennett, Garbuz, and Marten \cite{bgm} as a key ingredient to reduce the computation.

\section{There Are Only Two Goormaghtigh Primes Less Than $10^{\uplim}$}

Our approach is to use lower bounds on $m$, $n$, and $y$ in
\eqref{eq:goorm} to significantly limit the number of cases less than
$10^\uplim$ that we need to check.

The following is part of Theorem 4 from \cite{bgm}.
\begin{theorem}
\label{thm:bgm}
The only solutions to Equation \eqref{eq:goorm} with $\gcd(m-1,n-1)>1$ and $m\le 50$ have $N=31$ or $N=8191$.
\end{theorem}

For Goormaghtigh primes, we must have $m$ and $n$ odd primes, and, except for the two known solutions, $m\ge53$.
Theorem 2 from Bennett, Gherga, and Kreso \cite{bgk} rules out further examples with $n=3$ or $n=5$ when $\gcd(m-1,n-1)>1$, which is true in the prime case.
Theorem 3 from \cite{bgm} rules out further examples with $y\le 10^5$.
The finitely many possibilities with $x=2$ are checked separately.  In order
to show that there are no new Goormaghtigh primes below $10^{\uplim}$, we employ
the following algorithm for $x\ge3$. Let $f_k(x)=\frac{x^k-1}{x-1}$.
\begin{enumerate}
\item We perform a precomputation similar to the one in \cite{fgg}.
For a list of small primes $\{p_i\}$, we compute the values of $f_q(b)$ for all $2\le b < p_i$ and $1\le q < p_i$.
\item Generate a list of possible values of $m$. We know $m$ must be prime. From Theorem \ref{thm:bgm}, we have $m\ge 53$, and because $x\ge3$ in the range considered and $f_m(x)>x^{m-1}$, we have $m<1+\log{10^{\uplim}}/\log{3}\approx 1468$.
\item For each such $m$, examine every integer $x>2$ with $x<10^{\uplim/(m-1)}$.
\item Because $y>10^5$ and $f_n(y)>y^{n-1}$, the inequality
$f_n(y)<10^{\uplim}$ implies $5(n-1)<\uplim$.  We examine each prime
$n\le139$ such that $m>n\ge7$; we exclude
the case $m-1=3(n-1)$ by Theorem 6 of \cite{bgm}.
\item For each remaining triple $(m,n,x)$ and each precomputed prime $p_i$,
let $a_i=f_m(x)\pmod {p_i}$.  If $x\not\equiv0,1\pmod {p_i}$, we use the
precomputed table to look up $a_i$, with $x\bmod p_i$ as the base index and
$m\bmod(p_i-1)$ as the exponent index.  For $x\equiv0\pmod {p_i}$ and
$x\equiv1\pmod {p_i}$, we use $a_i\equiv1\pmod {p_i}$ and
$a_i\equiv m\pmod {p_i}$, respectively.  For
$n'\equiv n\pmod {p_i-1}$, we see if $a_i$ is a possible value of
$f_{n'}(y)$, including when $a_i=0$.  To handle the case
$y\equiv0\pmod {p_i}$, we check if $a_i\equiv1\pmod {p_i}$.  Finally, if
$y\equiv1\pmod {p_i}$, we have $f_n(y)\equiv n\pmod {p_i}$, so we check if
$a_i\equiv n\pmod {p_i}$.  If none of these conditions hold, there is no
solution modulo $p_i$, and thus none over the integers.  If one holds, we
repeat with the next $p_i$.  We check every triple that survives all 60
primes against the original equation using exact integer arithmetic.

\end{enumerate}

The corrected computation for $x\ge3$ up to $10^{\uplim}$ took 1 hour,
10 minutes, and 36 seconds using 112 processes on a machine with two 64-core
AMD EPYC 7742 processors.  The separate $x=2$ calculation checked all 9,997
admissible pairs $(m,n)$.

Note that we have, in fact, shown a slightly stronger result, that there are
no new Goormaghtigh numbers below $10^{\uplim}$ where both exponents are
prime.

\section{A Conditional Result}

Recall the abc conjecture.

\begin{conjecture}
The \textit{abc conjecture} of Oesterl\'e \cite{Oes} and Masser
\cite{Masser} states that if $a$, $b$, and $c$ are relatively prime
integers such that $a + b = c$, then for any $\epsilon >0$, only
finitely many $(a,b,c)$ fail to satisfy the inequality
\[ c < \rad(abc)^{1+\epsilon}.\]
\end{conjecture}

Carl Pomerance suggested an argument that gives the following.

\begin{theorem}
Assuming the abc conjecture, there are only finitely many Goormaghtigh numbers where neither representation is of length three or four.
\end{theorem}

\begin{proof}
Rewrite \eqref{eq:goorm} as
$$(x^m-1)(y-1)=(y^n-1)(x-1).$$
Rearranging,
$$y^n(x-1)+(y-x)=x^m(y-1).$$
Let $$g=\gcd(y^n(x-1),y-x,x^m(y-1)).$$
Let $a=y^n(x-1)/g$, $b=(y-x)/g$ and $c=x^m(y-1)/g$.
Then $a$, $b$, and $c$ are coprime, and we have $a+b=c$, with $b>0$ because
$y>x$.

Because $g\mid(y-x)$ we have
$g\le y-x\le y-1$, so the largest term satisfies
$$c=\frac{x^m(y-1)}{g}\ge x^m.$$
On the other hand,
$$\rad(abc)\le x(x-1)y(y-1)(y-x)<x^2y^3.$$
By \eqref{eq:goorm}, $x^{m-1}<N<2x^{m-1}$ and $y^{n-1}<N<2y^{n-1}$, so
$y^{n-1}<2x^{m-1}$. Because $n\ge 5$, we have $3/(n-1)\le 3/4$, and hence
$$y^3=(y^{n-1})^{3/(n-1)}<(2x^{m-1})^{3/(n-1)}\le 2^{3/4}x^{3(m-1)/4},$$
so that
$$\rad(abc)<2^{3/4}x^{5/4+3m/4}.$$

Let $\epsilon=1/25$. As neither representation is of length three or four, $m>n\ge 5$,
so $m\ge 6$, and
$$
\rad(abc)^{1+\epsilon}<2^{39/50}x^{(65+39m)/50}.
$$
If $m\ge10$, then the right-hand side is less than $x^m\le c$ because
$x\ge2$.  For each $m\in\{6,7,8,9\}$, the same conclusion holds
once $x$ exceeds a constant depending on $m$.  The remaining choices
of $m$ and $x$ are finite; then $n<m$ and $f_n(y)=f_m(x)$ leave only
finitely many possibilities.  Therefore, apart from finitely many
quadruples, every solution under consideration gives a coprime triple
$(a,b,c)$ satisfying
$$
c>\rad(abc)^{1+\epsilon}.
$$
The abc conjecture permits only finitely many such triples.
It remains to check that a fixed triple has only finitely many preimages.
Indeed,
$$
\frac{a}{b}=\frac{y^n(x-1)}{y-x}>y^{n-1},
$$
because $y(x-1)>y-x$ for $x\ge2$.  Thus $a/b$ bounds both $y$ and
$n$.  There are then only finitely many $x<y$, and, for fixed $x$,
$f_m(x)$ is strictly increasing in $m$.  Therefore there are only finitely
many Goormaghtigh numbers with $n\ge5$.
\end{proof}

Combining Theorem 2 with Theorem 2 from \cite{bgk}, which eliminates the
length-3 case for primes, gives the following.

\begin{corollary}
Assuming the abc conjecture, there are only finitely many Goormaghtigh primes.
\end{corollary}

\vskip20pt\noindent\textbf{Code and formal verification.} The corrected C
implementation, independent checking scripts, computational data,
and a Lean~4 formalization are available at
\url{https://github.com/31and8191/goor1/}.  The Lean development verifies the
corrected abc argument.

Separately, for the bounded computation, the Lean development verifies that
the certified $x\ge3$ intervals cover the full search range, checks the exact
rejection of every logged $x\ge3$ survivor, and verifies the complete $x=2$
branch.  The cited reductions remain explicit hypotheses of
the bounded theorem; for the $x\ge3$ branch, the theorem also assumes
that the C execution emitted every tuple surviving the formal sieve.  The
$x=2$ branch has no corresponding execution premise.

\vskip20pt\noindent{\bf Acknowledgements.} The author is grateful to Mike Bennett for helpful e-mail exchanges and providing copies of code, to Hester Graves for helpful conversations and checking the Mersenne primes, to Xander Faber and Mark Morgan for comments on an earlier draft of this paper, and to Gregory Minton for suggesting ruling out solutions locally. The author also thanks Christian Ballot, whose Math Review of the published version identified an error in the original proof of Theorem~2.

\vskip20pt\noindent{\bf Statement on AI use.} The preparation of this corrected version, including the revised proof and computational verification, used the language models Claude Opus 4.8 (Anthropic) and GPT-5.6 Sol (OpenAI) under the author's direction. The author has checked the argument and computation and takes responsibility for the contents.

\end{document}